\documentclass[12pt,leqno]{article}
\usepackage{amssymb,amsfonts,amsmath,amsthm,amscd,mathrsfs,authblk}
\usepackage[alphabetic]{amsrefs}
\usepackage{psfrag}
\def\IN{{\mathbb N}}

\def\dsl{\textstyle\sum\limits}

\def\pr{\textstyle\prod\limits}

\def\n{\noindent}
\def\dis{\displaystyle}
\def\r{\rightarrow}

\def\wt{\widetilde}

\def\o{\omega}
\def\O{\Omega}

\def\cN{{\cal N}}

\def\cZ{{\cal Z}}

\def\cO{{\cal O}}

\def\cZ{{\cal Z}}
\def\cP{{\cal P}}
\def\FF{{\mathbb{F}}}

\dimendef\dimen=0

\newtheorem{theorem}{Theorem}[section]
\newtheorem{lemma}[theorem]{Lemma}
\newtheorem{corollary}[theorem]{Corollary}
\newtheorem{proposition}[theorem]{Proposition}
\newtheorem{definition}[theorem]{Definition}
\newtheorem{remark}[theorem]{Remark}

\begin{document}
\title{\Large \bf Topological finite generation of compact open subgroups of universal groups}

\author[1]{Marc Burger}
\author[2]{Shahar Mozes
\footnote{M.B. thanks the hospitality of the IMS of the National University of Singapore, S.M. 
acknowledges the support of ISF grants 1003/11 and 2095/15. 
We both thank the hospitality of the University of Newcastle
 and the ANSI program "Winter of disconnectedness". We thank
 the Isaac Newton Institute for Mathematical Sciences for support and hospitality during the  
 program on "Non-positive curvature group action and cohomology".
 This work was supported by EPSRC Grant Number EP/K032208/1.
}
}
\affil[1]{ETH Z\"urich}
\affil[2]{The Hebrew University of Jerusalem}
\date{}
\maketitle
This paper is part of our ongoing study of the structure of co-compact irreducible lattices in the product of two regular trees, see \cite{BuMo1}, \cite{BuMo2}, \cite{BuMoZ}, \cite{Ra1}, \cite{BuMo3}. The specific result obtained here, see Corollary \ref{cor0.2}, is motivated by the following question: Let $T_d = (V,E)$ be the $d$-regular tree, $F < {\rm Sym}_d$ a permutation group on $d$ letters and $U(F) < {\rm Aut} \,T_d$ the universal group attached to it; see \cite{BuMo1}, for definitions and properties. When is $U(F)$ the closure of the projection of a co-compact lattice $\Gamma < U(F) \times G$, where $G$ is locally compact, compactly generated? Work of D. Rattaggi (using \cite{BuMo2}) implies that this is so when $F= A_{2 n}$ is the alternating group with $2n \ge 6$, or $F = M_{12}$ is the $12$-Matthieu group, or $F = ASL_3(2$) is the special affine group on $\FF^3_2$; in addition there are co-compact lattices in $U(A_{2n}) \times U(A_{2m})$ with dense projections for $n \ge 13$, $m \ge 9$.

\medskip
Our starting observation is that if $\Gamma < U(F) \times G$ is a co-compact lattice with $G$ as above, then any compact open subgroup of $U(F)$ is topologically finitely generated. Our main result gives a characterization of the permutation groups $F$ for which this property holds. In fact, this result will follow from a similar result concerning the finite generation of iterated wreath products. More precisely, let $D$ be a finite set and $L < {\rm Sym} (D)$ a permutation group. Let $L_n: = L  \wr \dots \wr L$ be the $n$'th iterated wreath product of $L$ acting on $D$, and $L_\infty: = \lim\limits_{\leftarrow} L_n$ the projective limit.
Note that $L_\infty$ has a natural action on the rooted tree whose vertices are labeled by words over $D$ (the root is the empty word).

\begin{theorem}\label{theo0.1}
The profinite group $L_\infty$ is topologically finitely generated if and only if the two following conditions are satisfied:
\begin{itemize}
\item[{\rm a)}] $L$ is perfect.

\item[{\rm b)}] Every $L$-orbit in $D$ has at least two elements.
\end{itemize}

\medskip\n
If these conditions are satisfied, then $L_\infty$ is positively finitely generated.
\end{theorem}

\medskip
This result was announced in \cite{Mo98}. A version of the proof was written up in the thesis of O. Amann \cite{Am} following a manuscript by the authors. The proof presented here substantially simplifies the original one.

Applied to our situation, let $F_1 = {\rm Stab}_F(1)$, and $D = \{2, \dots, d\}$. Then for every $m \ge 1$, the pointwise stabilizer in $U(F)$ of the ball or radius $m \ge 1$ centered at any vertex is a finite direct product of $(F_1)_\infty = \lim_n (F_1 \wr \dots \wr F_1)$ to which our theorem applies and gives

\begin{corollary}\label{cor0.2}
The profinite group obtained by considering the stabilizer in $U(F)$ of a vertex is topologically finitely generated if and only if the following two conditions are satisfied:
\begin{itemize}
\item[{\rm a)}] $F_1$ is perfect.

\item[{\rm b)}] Any $F_1$ orbit in $\{2, \dots , d\}$ has at least two elements.
\end{itemize}

\medskip\n
Moreover, when those conditions are satisfied, any compact open subgroup of $U(F)$ is positively finitely generated.
\end{corollary}

\medskip
Recall that a compact group $K$ is positively finitely generated if for some $\ell \ge 1$, the Haar measure of the subset of $\ell$-tuples in $K^\ell$ generating a dense subgroup of $K$ is positive. For background information concerning this property, we refer to \cite{Mann1}.
The question of finite generation of various infinite iterated wreath products was considered by several authors. See \cite{Bhatt}, \cite{Quick1}, \cite{Quick2}, \cite{Bondarenko}, \cite{Vann}. All these works consider wreath products of transitive group actions.

\medskip
Now we shortly discuss the necessity of the conditions in Theorem \ref{theo0.1} and to that end we start by fixing some notations used throughout this article. The permutation groups $L_n$ acting on $S_n : = D^n$ are defined inductively by $L_1 = L < {\rm Sym}(D)$ and $L_{n+1} = L_n \ltimes L^{S_n}$ where an element in $L_{n+1}$ is a pair $(x,f)$, $f$: $S_n \r L$ a map, and the product structure reads $(x_1,f_1)(x_2,f_2) = (x_1 x_2, f_1^{x_2} f_2)$, and $f^{x_2}(w) = f(x_2 w)$. The action of $L_{n+1}$ on $S_n \times D$ is given by $(x,f)(v,j) = (x(v), f(v) j)$.

\medskip
Now we show the necessity of the two conditions.

\begin{itemize}
\item[a)] If $\pi$: $L \r A$ is an abelian quotient of $L$, then
\begin{align*}
L_n \ltimes L^{S_n} & \longrightarrow A
\\
(x,f) & \longmapsto \pr_{v \in S_n} \pi\big(f(v)\big)
\end{align*}

\n
is a surjective homomorphism; by induction this implies that $A^n$ is a quotient of $L_n$ and hence $L_\infty$ admits $A^\IN$ as quotient. Thus, if $L$ is not perfect, $L_\infty$ is not topologically finitely generated.

\medskip
\item[b)] If $j \in D$ is $L$-fixed, then $(j,\dots,j) \in S_n$ is $L_n$-fixed and $\prod_{v \in S_n \atop v \not= (j,\dots,j)} L$ is a normal subgroup of $L_{n+1} = L_n \ltimes L^{S_n}$ with quotient $L_n \times L$. By recurrence this implies that $L^n$ is a quotient of $L_n$ and hence $L^\IN$ of $L_\infty$. But if $L \not= (e)$, the $L^\IN$ is not topologically finitely generated and hence $L_\infty$ not as well.
\end{itemize}

\medskip
The plan of the paper is the following: the proof of the converse of Theorem \ref{theo0.1} follows a strategy devised by Bhattacharjee in \cite{Bhatt}. Namely, if $p_k(G)$ denotes the probability for a $k$-tuple to generate the finite group $G$, we show that under the hypothesis of Theorem \ref{theo0.1}, $\lim_n p_k(L_n) > 0$ for some $k$. For this we use a result of Bhattacharjee (see \cite{Bhatt}), which we recall below, relating $p_k(L_{n+1})$ to $p_k(L_n)$ modulo a multiplication factor $(1 - \zeta_{L_{n+1}/L_n}(k-1))$ which is defined in terms of the conjugacy classes of maximal subgroups of $L_{n+1}$ surjecting onto $L_n$. The main work consists then in classifying these conjugacy classes.

\section{A result of M. Bhattacharjee}

We recall for the convenience of the reader the following:

\begin{proposition}\label{prop1.1} (\cite{Bhatt})
Let $\pi$: $Y \r X$ be a surjective homomorphism of finite groups and $k \ge 1$. Then
\begin{equation*}
p_k (Y) \ge \Big(1 - \dsl_{[M] \in G(Y|X)} \; \dis\frac{1}{[Y\! : \!M]^{k-1}}\Big) \,p_k(X)
\end{equation*}

\n
where the sum is over the set $G(Y|X)$ of $Y$-conjugacy classes of proper maximal subgroups of $Y$ surjecting onto $X$.
\end{proposition}

\begin{proof}
We have $p_k(Y) = p_k(Y|X) \,p_k(X)$, where $p_k(Y|X)$ is the probability for a $k$-tuple to generate $Y$ given that its image in $X$ generates $X$. Then
\begin{equation*}
p_k(Y|X) = 1 - q_k(Y|X)
\end{equation*}

\n
where $q_k(Y|X)$ is the probability of the opposite event. Observing that a non-generating $k$-tuple is always contained in a proper maximal subgroup $M < Y$, we obtain:
\begin{align*}
q_k(Y|X) & \le \dsl_{M\lneq Y \;{\rm maximal}} \;\dis\frac{|M|^k}{|Y|^k}
\\[1ex]
& = \dsl_{[M] \in G (Y|X)} [Y\!:  \cN_Y(M)] \cdot \dis\frac{|M|^k}{|Y|^k}
\\[1ex]
&\le \dsl_{[M] \in G(Y|X)} \; \dis\frac{1}{[Y\!:\!M]^{k-1}}.
\end{align*}
\end{proof}

Denoting by $\zeta_{Y|X}(s) = \sum_{[M] \in G(Y|X)}\, \frac{1}{[Y \cdot M]^s}$, we conclude:

\begin{corollary}\label{cor1.2}
Assume that for some $k_1 \ge 2$, $\sum^\infty_{k=1} \zeta_{L_{n+1}/L_n}(k_1 - 1) < + \infty$. Let $n_1$ be such that $\zeta_{L_{n_1 + 1}/L_{n_1}}(k_1 - 1) < 1$ and let $k_2 \ge 2$ be such that $L_{n_1}$ is generated by $k_2$ elements. Then $L_\infty = \lim\limits_\leftarrow L_n$ is positively $k$-generated for $k \ge \max(k_1,k_2)$.
\end{corollary}

\begin{proof}
For every $k \ge k_1$ we have (Proposition \ref{prop1.1})
\begin{equation*}
\lim\limits_{n \r \infty} p_k(L_n) \ge \pr_{n \ge n_1} \big(1 - \zeta_{L_{n+1}/L_n}(k-1)\big) \cdot p_k (L_{n_1}).
\end{equation*}
Now observe that
\begin{equation*}
\pr_{n \ge n_1} \big(1 - \zeta_{L_{n+1} / L_n}(k-1)\big) \ge \pr_{n \ge n_1}  \big(1 - \zeta_{L_{n+1}/L_n}(k_1-1)\big) > 0
\end{equation*}

\medskip\n
and conclude by choosing $k \ge \max(k_1,k_2)$.
\end{proof}

\section{Maximal subgroups in wreath products}

Let $X$ be a finite group, $\Omega_1,\dots,\Omega_t$ transitive (non-empty) $X$-sets, $t \ge 1$ and $B_1,\dots, B_t$ non-trivial perfect groups; these will be our standing assumptions throughout this section. Let $Y = X \ltimes (B_1^{\Omega_1} \times \dots \times B_t^{\Omega_t})$ be the semi-direct product, where the action of $X$ is by permuting factors, and $\pi$: $Y \r X$ the projection map.

\begin{definition}\label{def2.1}
A standard normal subgroup of $Y$ is a subgroup of the form
\begin{equation*}
\pr^t_{i=1} N_i^{\O_i}
\end{equation*}
where $N_i \vartriangleleft B_i$
\end{definition}

Observe that every subgroup $M<Y$ contains a unique maximal standard normal subgroup of $Y$.

\begin{definition}\label{def2.2}
A subgroup $M < Y$ is clean if it contains no non-trivial standard normal subgroup.
\end{definition}

The following proposition summarizes the ingredients needed in the proof of Theorem \ref{theo0.1}; this proposition is a corollary of more precise statements proven in this section.

\begin{proposition}\label{prop2.3}
Let $M < Y$ be a clean, proper, maximal subgroup such that $\pi(M) = X$.

\medskip
Then one of the following holds:

\begin{itemize}
\item[{\rm 1)}] $t = 2$, $B_1, B_2$ are non-abelian simple and

\begin{itemize}
\item[{\rm (a)}] $M \cap (B_1^{\Omega_1} \times B_2^{\Omega_2})$ is the graph of an isomorphism $B_1^{\Omega_1} \r B_2^{\Omega_2}$.

\item[{\rm (b)}] $M = \cN_Y (M \cap (B_1^{\Omega_1} \times B_2^{\Omega_2}))$.
\end{itemize}

In particular, there are at most $|\Omega_1| \cdot |{\rm Out}(B_1)|$ conjugacy classes of such subgroups, and $[Y \!:\! M] \ge |B_1|^{|\Omega_1|}$.

\medskip
\item[{\rm 2)}] $t = 1$, $M \cap B^\Omega \not= (e)$ and $pr_w (M \cap B^\Omega) = B$ for all $w \in B$. There is a normal subgroup $U \vartriangleleft B$ which is a product $U = T^r$ where $T$ is non-abelian simple. Moreover,

\begin{itemize}
\item[{\rm (a)}] $M\cap U^\Omega$ is a product of subdiagonals of $T^{r \cdot \Omega}$ corresponding to an $X$-invariant block decomposition of $r \cdot \Omega$.

\item[{\rm (b)}] $M = \cN_Y (M \cap U^\Omega)$.

\item[{\rm (c)}] For any given $X$-invariant block decomposition of $r\cdot \Omega$ there are at most $|{\rm Out}(T)|^{r|\Omega |}$ conjugacy classes of such subgroups.

\item[{\rm (d)}] There are most $(2 | \Omega |)^{r -1} a^r_\Omega$ such block decompositions of $r\cdot \Omega$ where $a_\Omega$ is the number of $X$-invariant block decompositions of $\Omega$.

\item[{\rm (e)}] $[Y: M] \ge |T|^{\frac{r \cdot |\Omega |}{2}} = |U|^{\frac{|\Omega |}{2}}$.
\end{itemize}

\medskip
\item[{\rm 3)}]
$t=1$, $M\cap B^{\Omega}\neq (e)$ and $pr_\omega (M\cap B^{\Omega})$ is a proper subgroup of $B$. Up to conjugation $pr_{\omega}(M\cap B^{\Omega})=T$ for all $\omega\in \Omega$,
$$M = \cN_{Y}(T^{\Omega})
$$
and $[Y:M] \ge [B:T]^{|\Omega|}$.
\medskip
\item[{\rm 4)}] $t = 1$, $M \cap B^\Omega = e$.
\end{itemize}
\end{proposition}

Now we state and prove several lemmas which together will imply the above proposition. We recall that we will work under the above standing assumptions on the objects $X, B_1,\dots,B_t$, $\Omega_1,\dots , \Omega_t$.

\begin{lemma}\label{lem2.4}
Assume $M < Y$ is proper, maximal, clean. Then $\cZ(B_i)=(e)$ for all $1 \le i \le t$.
\end{lemma}

\begin{proof}
Assume that for some $j$, $\cZ(B_j) \not= (e)$. Since $M$ is clean and maximal, \linebreak $M \cdot \cZ(B_j)^{\Omega_j} = Y$ and since $B^{\Omega_j}_j \supset \cZ(B_j)^{\Omega_j}$ this implies ${B_j^{\Omega_j} = (M\cap B_j^{\Omega_j})\cZ(B_j)^{\Omega_j}}$. Since $B_j$ is perfect we obtain:
\begin{equation*}
B_j^{\Omega_j} = [B_j^{\Omega_j}, B_j^{\Omega_j}] = [M \cap B_j^{\Omega_j}, M \cap B_j^{\Omega_j}] \subset M
\end{equation*}

\n
which contradicts the assumption that $M$ is clean.
\end{proof}

\medskip
Given a subgroup $M < Y$ we let $M^0:= \bigcap_{y \in Y} y\,My^{-1}$ denote the kernel of the $Y$-action on $Y/M$. Also, given $N_i < B_i$, $w \in \Omega_i$, let $N_i(\o) \subset \prod^t_{j=1} B^{\Omega_j}_j$ denote the subgroup whose only non-identity component is at $\o \in \Omega_i$ and equals $N_i$.

\begin{remark}\label{rem2.5} \rm
If $M < Y$ is clean, $\pi (M) = X$ and $N_i  \vartriangleleft B_i$, then $M^0 \cap N_i(w) = (e)$ for all $\o \in \Omega_i$. Indeed, otherwise since $\pi (M) = X$, $M$ would contain the normal subgroup $\prod_{\o \in \Omega_i} M^0 \cap N_i(\o)$ which is standard.
\end{remark}

\begin{lemma}\label{lem2.6}
Assume $M < Y$ is proper, clean, maximal and $\pi(M) = X$. Then
\begin{equation*}
M^0 \cap \pr^t_{i=1} B_i^{\Omega_i} = (e).
\end{equation*}
\end{lemma}

\begin{proof}
If $M^0 \cap \prod^t_{i=1} B_i^{\Omega_i} \not= (e)$, then let $j$ and $\o \in\Omega_j$ be such that
\begin{equation*}
N_j : = pr_\o \Big(M^0 \cap \pr^t_{i=1} B_i^{\Omega_i}\Big) \not= e.
\end{equation*}
Then $N_j \vartriangleleft B_j$ and
\begin{equation*}
[N_j(\o), B_j(\o)] = \Big[M^0 \cap \pr^t_{i=1} B_i^{\Omega_i}, \,B_j(\o)\Big] \subset M^0 \cap B_j(\o).
\end{equation*}

\n
Since by Lemma \ref{lem2.4}, $\cZ(B_j) = (e)$ we must have $[B_j(\o), N_j(\o)] \not= (e)$ and hence $M^0 \cap B_j(\o) \not= (e)$. But this contradicts the cleanness assumption by Remark \ref{rem2.5}.
\end{proof}

\begin{lemma}\label{lem2.7}
Let $M < Y$ be clean, proper, maximal with $\pi(M) = X$. If $Y_1,Y_2$ are non-trivial normal subgroups of $Y$ contained in $\prod^t_{i=1} B_i^{\Omega_i}$ such that
\begin{equation*}
Y_1 \cap Y_2 = (e) .
\end{equation*}
Then both $Y_1,Y_2$ act transitively and regularly on $Y/M$. In particular $ t=2$ \hbox{or $1$.}
\end{lemma}

\begin{proof}
The $Y$-action on $Y/M$ is primitive and (Lemma \ref{lem2.6}) the subgroup $\prod^t_{i=1} B^{\Omega_i}_i$ is faithful.
The action of each of the normal subgroups $Y_i$ is transitive and faithful. Since they commute each of them must act freely and hence regularly. This implies that there cannot be three normal subgroups of $Y$ contained in $\prod^t_{i=1} B_i^{\Omega_i}$ such that every two of them intersect only at the identity. Thus $t\in\{1,2\}$.
\end{proof}

\begin{lemma}\label{lem2.8}
Assume that $M < Y$ is proper, clean, maximal with $\pi(M) = X$ and assume $t = 2$. Then $B_1,B_2$ are non-abelian simple, $M \cap (B_1^{\Omega_1} \times B_2^{\Omega_2})$ is the graph of an isomorphism
$B_1^{\Omega_1}\to B_2^{\Omega_2}$ and
\begin{equation*}
M = \cN_Y \big(M \cap (B_1^{\Omega_1} \times B_2^{\Omega_2})\big).
\end{equation*}
In addition, $[Y\!:\!M] \ge |B_1|^{|\Omega_1|} = |B_2|^{|\Omega_2|}$.
\end{lemma}

\begin{proof}
It follows from Lemma \ref{lem2.7} that both $B_1^{\Omega_1}$ and $B_2^{\Omega_2}$ act transitively and regularly on $Y/M$. Hence, if $(e) \not= N_i \vartriangleleft B_i$, then $N_i^{\Omega_i}$ being non-trivial, normal in $Y$, acts transitively on $Y/M$, hence $|N_i^{\Omega_i}| = |B_i^{\Omega_i}|$ and $N_i = B_i$ which shows that $B_i$ is simple.

\medskip
Next, $ B_1^{\Omega_1}\cdot M = Y$, hence $B_1^{\Omega_1} \times B_2^{\Omega_2} = B^{\Omega_1}_1 \cdot [M \cap (B_1^{\Omega_1} \times B_2^{\Omega_2})]$ and thus $pr_2 (M \cap (B_1^{\Omega_1} \times B_2^{\Omega_2})) = B_2^{\Omega_2}$. Similarly, we have $pr_1(M \cap (B_1^{\Omega_1} \times B_2^{\Omega_2})) = B_1^{\Omega_1}$. Since we also have $M \cap B_1^{\Omega_1} = M \cap B_2^{\Omega_2} = (e)$, we deduce that $M \cap (B_1^{\Omega_1} \times B_2^{\Omega_2})$ is the graph of an isomorphism $B_1^{\Omega_1} \r B_2^{\Omega_2}$. Finally, we have
\begin{equation*}
M \subset \cN_Y\big(M \cap (B_1^{\Omega_1} \times B_2^{\Omega_2})\big)
\end{equation*}
and if equality does not hold then $M \cap (B_1^{\Omega_1} \times B_2^{\Omega_2})$ is normal in $Y$ and hence equal to $M^0 \cap  (B_1^{\Omega_1} \times B_2^{\Omega_2})$, contradicting Lemma \ref{lem2.6}. The lower bound on $[Y\!:\!M]$ is clear.
\end{proof}

\medskip
Next, we need to estimate the number of $Y$-conjugacy classes of maximal subgroups as in Lemma \ref{lem2.8}.
\begin{lemma}\label{lem2.9}
The number of conjugacy classes of maximal subgroups as in Lemma \ref{lem2.8} is bounded by
\begin{equation*}
|{\rm Out}(B_1)| \cdot |\{v_2 \in \Omega_2: {\rm Stab}_X(v_2) = {\rm Stab}_X(v_1)\}|
\end{equation*}
where $v_1 \in \Omega_1$ is some chosen element.
\end{lemma}

\begin{proof}
According to Lemma \ref{lem2.8} the subgroup $M$ is determined by an isomorphism $F$: $B_1^{\Omega_1} \r B_2^{\Omega_2}$. Since $B_1,B_2$ are non-abelian simple such an isomorphism is given by the following data:

\begin{itemize}
\item[(1)] a bijection $\sigma$: $\Omega_1 \r \Omega_2$
\item[(2)] for every $w \in\Omega_1$ an isomorphism $\varphi_w$: $B_1 \r B_2$.
\end{itemize}

\medskip
Where
\begin{equation*}
F(f)(\big(\sigma(w)\big) = \varphi_w\big(f(w)\big), \quad \forall w \in \Omega_1, \quad \forall f \in B_1^{\Omega_1}.
\end{equation*}

\n
Then $M = \cN_Y({\rm graph}(F))\supset {\rm graph}(F)$ and the condition that $M$ surjects onto $X$ together with $M \cap B_2^{\Omega_2} = (e)$ gives that for every $x \in X$ there is a unique $h_x \in B_2^{\Omega_2}$ with $(x,e,h_x) \in M$. From this we deduce
\begin{equation}\label{2.1}
h_{x_1 x_2} = h_{x_1}^{x_2} \cdot h_{x_2} \quad \forall x_1,x_2 \in X.
\end{equation}
Next, using that $(x,e,h_x)$ normalizes the graph of $F$, we get for all $f \in B_1^{\Omega_1}$, $x \in X$,
\begin{equation*}
F(f^{x^{-1}}) = h_x^{x^{-1}} F(f)^{x^{-1}} (h_x^{x^{-1}})^{-1} .
\end{equation*}

\n
Evaluating this at $\sigma(\o)$ and for maps $f\in B_1^{\Omega_1}$ whose support is a singleton, we get
\begin{align}
&x \sigma(w) = \sigma(xw) \quad \forall w \in \Omega_1, \quad \forall x \in X.  \label{2.2}
\\[2ex]
&\varphi_{xw}(b) = h_x \big(\sigma(w)\big) \,\varphi_w (b) \;h_x\big(\sigma(w)\big)^{-1}\; \forall b \in B_1, \;\forall w \in \Omega_1, \; \forall x \in X. \label{2.3}
\end{align}

\n
Now fix a point $w_0 \in \Omega_1$ and observe that $h_x(\sigma(w_0)) \in \cZ(B_2) = (e)$ for all $x \in {\rm Stab}_X(w_0)$. Together with (\ref{2.1}) this implies that $f_2$: $\Omega_2 \r B_2$ given by $f_2(v) = h_{x_2}(\sigma(w_0))$ where $x_2 \in X$ is any element with $x_2 \,\sigma(w_0) = v$, is a well defined map.

\medskip
Rewriting the equation (\ref{2.1}) in terms of $f_2$, we get
\begin{equation*}
h_x = f_2^x \cdot f^{-1}_2.
\end{equation*}
Using this one computes that the group $(e,e,f_2) \, M(e,e,f_2)^{-1} = M'$ corresponds to an isomorphism $F': B_1^{\Omega_1} \r B_2^{\Omega_2}$ given by

\begin{itemize}
\item[1)] the same equivariant bijection $\sigma$: $\Omega_1 \r \Omega_2$;
\item[2)] a family of isomorphism $\varphi'_w$: $B_1 \r B_2$ which all coincide.
\end{itemize}

\medskip
From this, one readily deduces the upper bound stated in Lemma \ref{lem2.9}.
\end{proof}

\medskip
Now we turn to the situation where $t=1$, that is $Y = X \ltimes B^\Omega$, where we have set $B = B_1$, $\Omega = \Omega_1$. This splits into four different cases which we will analyze now. Our standing assumption on $M < Y$ is that it is a proper, clean, maximal subgroup which surjects onto $X$.

\begin{lemma}\label{lem2.10}
Assume that $M \cap B^\Omega \not= (e)$, $pr_w(M \cap B^\Omega) = B$ $\forall \o \in \Omega$, and that for every minimal normal subgroup $N \vartriangleleft B$, $M \cap N^\Omega = (e)$. Then $B$ is the product of two isomorphic non-abelian simple groups. Furthermore:
\begin{equation*}
M = \cN_Y(M \cap B^\Omega).
\end{equation*}
\end{lemma}

\begin{proof}
We claim that every non-trivial minimal normal subgroup $N \vartriangleleft B$ is a direct factor. For $\o \in \Omega$ let $B_\o : = \{f \in B^\Omega$: $f(\nu) \in N \; \forall \nu \not= w\}$. From $Y = M \cdot N^\Omega$ we deduce $B_\o = (M \cap B_\o) \,N^\Omega$ and hence $B= pr_\o(M \cap B_\o)\,N$. Observe that $M \cap B_\o \vartriangleleft M \cap B^\O$ and thus $pr_\o(M \cap B_\o) \vartriangleleft B$. Finally, $pr_\o(M \cap B_\o) \cap N = pr_\o(M \cap N^\Omega) = (e)$, which shows that $N$ is a direct factor of $B$. This implies using Lemma \ref{lem2.7} that $B$ is either non-abelian simple or the product of two such groups.
The assumption that for any minimal normal subgroup $N\vartriangleleft B$ $M\cap N^{\Omega}=(e)$ implies that $B$ is not simple.

\medskip
If the last assertion of the lemma were not to hold, then $M \cap B^\O \vartriangleleft Y$ and hence ${M^0 \cap B^\O = M \cap B^\O \not= (e)}$, contradicting Lemma \ref{lem2.6}.
\end{proof}

\begin{lemma}\label{lem2.11}
Assume that $M \cap B^\Omega \not= (e)$, $pr_\o(M \cap B^\Omega) = B \; \forall \o \in \Omega$ and there is $N \vartriangleleft B$ minimal normal with $M \cap N^\Omega \not= (e)$. Then $N$ is a product of (isomorphic) non-abelian simple groups.

\begin{itemize}
\item[{\rm (a)}] $pr_w(M \cap N^\Omega) = N$ for all $w \in \Omega$.

\item[{\rm (b)}] $M = \cN_Y(M \cap N^\Omega)$.
\end{itemize}
\end{lemma}

\begin{proof}
Since $N^\O$ acts transitively, faithfully and not regularly on $Y/M$, it (and hence also $N$) cannot be abelian. Since $N$ is a minimal normal subgroup of $B$ it is therefore a product of isomorphic non-abelian simple groups.

\medskip
Next, we observe that $pr_\o(M \cap N^\O)$ is normal in $B = pr_\o(M \cap B^\O)$, and being non-trivial, we deduce $pr_w(M \cap N^\O) = N$ for all $w \in \O$. Again, if $M \cap N^\Omega \vartriangleleft Y$, then $M^0 \cap N^\Omega = M \cap N^\Omega \not= (e)$ contradicting Lemma \ref{lem2.6}.
\end{proof}

It follows from Lemma \ref{lem2.10} and Lemma \ref{lem2.11} that when $t = 1$, $M \cap B^\Omega \not= (e)$ and $pr_w (M \cap B^\Omega) = B$ for all $w \in \Omega$, there is a normal subgroup $U \vartriangleleft B$ which is a product $U = T^r$ where $T$ is non-abelian simple and $pr_w(M \cap U^\Omega) = U$ for all $w \in \Omega$. Since $T$ is non-abelian simple, this implies already the assertion 2) (a) in Proposition \ref{prop2.3}. Concerning 2) (b), we observe that, as usual, if $M \cap U^\Omega \vartriangleleft Y$, then $M^0 \cap U^\Omega = M \cap U^\Omega \not= (e)$ which contradicts Lemma \ref{lem2.6}.

\medskip
Next, we turn to assertions (c), (d) and (e) which will follow from the following discussion. The subgroup $M \cap U^\Omega = M \cap T^{r \cdot \Omega}$ projects onto $T$ for each $w \in r \cdot \Omega$, where we recall that $r \cdot \Omega$ denotes the disjoint union of $r$ copies of $\Omega$ with corresponding $X$-action. Since $\pi(M) = X$, the subgroup  $H = M \cap T^{r\cdot \Omega}$ is a product of subdiagonals of $T^{r\cdot \Omega}$ corresponding to a $X$-invariant block decomposition of $r \cdot \Omega$. Given a partition $\cP$ of $r \cdot \Omega$ such a subgroup $H$ is obtained in the following way: for each $x \in r \cdot \Omega$ there is an automorphism $\varphi_x$ of $T$ and $H$ is the direct product over $A \in \cP$ of the diagonal subgroup of $T^A$ given by
\begin{equation*}
\big\{ \big(\varphi_x(t)\big): t \in T, \; x \in A\big\}.
\end{equation*}

\bigskip\n
The index of $H$ in $T^{ r \cdot \Omega}$ is then $|T|^{e}$ where $e = \sum_{A \in \cP} (|A| -1)$.
Observe that $|A|\ge 2$ for every $A\in \cP$. Indeed otherwise for some $\omega \in \Omega$, $M$ would contain a $T$-factor of $U(\omega)$; but $pr_\omega(M\cap B^{\Omega})=B$ which implies
in the case of Lemma~\ref{lem2.11} where $U=N$
is a minimal normal subgroup of $B$, that $M \supset N(\omega)$ and hence since the projection of $M$ is $X$ it follows that $M\supset N^{\Omega}$ contradicting the assumption that $M$ is clean.
In the case of Lemma~\ref{lem2.10} where $U=B=T^2$, it follows that $M$ would contain a $T$ factor of $B(\omega)$. Hence since the projection of $M$ is $X$ it follows that $M\supset T^\Omega$ which since $T\vartriangleleft B$ is a standard normal subgroup, contradicting the assumption that $M$ is clean.
This implies $[Y: M] \ge [T^{r \cdot \Omega}:H] \ge |T|^{\frac{r |\Omega|}{2}}$ and establishes assertion (e) of Proposition \ref{prop2.3}. Concerning assertion (c), just observe that we can conjugate $H$ by inner automorphisms of $T^{r \cdot \Omega}$ so that for a given partition we have at most $|{\rm Out} (T)|^{r | \Omega |}$ conjugacy classes of such subgroups.

\medskip
What remains is to estimate the number of $X$-invariant partitions of $r \cdot \Omega$. To this end, let $\Omega_1,\Omega_2$ be transitive $X$-sets and set $\wt{\Omega} = \Omega_1 \cup \Omega_2$.

\bigskip
If $\cP$ is an $X$-invariant partition with $|\cP| = k$ such that some $A \in \cP$ intersects $\Omega_1$ and $\Omega_2$, then $\cP$ induces partitions into $k$ pieces of $\Omega_1$ and $\Omega_2$. Denoting by $a(k,\Omega_i)$ the number of $X$-invariant partitions with $k$ pieces of $\Omega_i$ and by
$a_{\Omega_i}$ the total number of $X$-invariant partitions of $\Omega_i$, we have that the number of $X$-invariant partitions of $\Omega_1 \cup \Omega_2$ is estimated by:
\begin{equation*}
\begin{array}{l}
a_{\Omega_1} \cdot a_{\Omega_2} + \dsl^{\min(|\Omega_1|,|\Omega_2|)}_{k=1} k\; a(k,\Omega_1) \, a(k,\Omega_2)\  \le
\\
\\
(1+m)a_{\Omega_1} a_{\Omega_2} \le 2m \,a_{\Omega_1} a_{\Omega_2}\,,
\end{array}
\end{equation*}

\medskip\n
where $m = \min (|\Omega_1|, |\Omega_2|)$ and the estimate in the first line uses the observation that to obtain a partition into k pieces of $\Omega_1\cup\Omega_2 $ whose pieces meet both $\Omega_i$'s one needs to partition each $\Omega_i$ into $k$ pieces and pair them. The pairing of these pieces is determined (by the transitivity of the action on each $\Omega_i$) by choosing for one piece of $\Omega_1$ a piece of $\Omega_2$.
Applying this inequality with $\Omega_1 = (r-1)\cdot \Omega$ and $\Omega_2=\Omega$ we get $a_{r \cdot \Omega} \le 2 | \Omega| \;a_{(r- 1) \Omega} \cdot a_\Omega$ and by recurrence,
\begin{equation*}
a_{r \cdot \Omega}  \le \big(2|\Omega|\big)^{(r-1)} (a_\Omega)^r.
\end{equation*}

\begin{lemma}\label{lem2.12}
Assume that $M \cap B^\O \not= (e)$ and $pr_\o(M \cap B^\O) \not= B$ for all $\o \in \O$. Then up to conjugating $M$, there is a subgroup $e \not= T \lneq B$ such that $M = \cN_Y(T^\O)$. We have $[Y\!:\! M] \ge [B\!:\!T]^{|\O|}$.
\end{lemma}

\begin{proof}
Let $T_\o: = pr_\o(M \cap B^\O)$, $\o \in \O$. Then, since $M$ surjects onto $X$, all the subgroups $T_\o$ are conjugated within $B$. In addition, $M$ normalizes $\prod_\o T_\o$: indeed, it is the unique smallest subgroup of $B^\O$ containing $M \cap B^\O$ and having direct product structure. Thus, modulo conjugating $M$, we may assume $T_\o = T \; \forall \o \in \O$, hence $M \subset \cN_T(T^\O)$. If they were not equal, we would have that $T \vartriangleleft B$, hence from $M \cdot T^\O = Y$ we would deduce $B^\O = (M \cap B^\O) \,T^\O$, but since $M \cap B^\O \subset T^\O$ this would imply $T = B$ which is a contradiction.
\end{proof}

\medskip
The last case left is when $M \cap B^\O = (e)$, that is, $M$ is a section of $X$ in $Y$. Concerning this, we cannot say anything in this degree of generality, and we will have to estimate by recurrence the number of sections of $L_n$ in groups of the form $L_n \ltimes  B^\cO$, where $\cO \subset S_n$ is an $L_n$-orbit.

\section{Estimating $\zeta_{L_{n+1}/L_n}(k)$}

In order to prove Theorem \ref{theo0.1}, we will use Corollary \ref{cor1.2} and show that for some $k \ge 1$,
\begin{equation*}
\dsl^\infty_{n=1} \zeta_{L_{n+1}/ L_n}(k)  < + \infty.
\end{equation*}
Thus we have to estimate
\begin{equation*}
\zeta_{L_{n+1}/ L_n}(k) = \dsl \; \dis\frac{1}{[L_{n+1}\!:\! M]^k} \, ,
\end{equation*}

\n
where the sum is over the conjugacy classes of proper, maximal subgroups $M < L_{n+1}$ surjecting onto $L_n$. To this end, let $M_s$ denote the maximal standard normal subgroup of $L_{n+1}$ contained in $M$ (see Definition \ref{def2.1}).
Let $I_n$ be the set of $L_n$ orbits in $S_n$.
Then
\begin{equation*}
L_{n+1} / M_s = L_n  \ltimes \pr_{v \in I_n} (B_v)^v \, ,
\end{equation*}

\n
where $B_v$ are quotients of $L$ and $M / M_s$ is a clean proper maximal subgroup which surjects onto $L_n$. Let us observe first that we cannot have $B_v = (e)$ for all $v \in I_n$ since otherwise $M_s = L^{S_n}$ which together with the hypothesis that $M$ surjects onto $L_n$ would imply that $M = L_{n+1}$ contradicting properness. Proposition \ref{prop2.3} implies then that there is either exactly one pair $v_1 \not= v_2$ of orbits for which $B_{v_1} \not=  (e)$ and $B_{v_2} \not= (e)$, or there is exactly one orbit $v \in I_n$ for which $B_v \not= (e)$. We are going to estimate the contribution to $\zeta_{L_{n+1}/L_n}$ according to the four different cases of Proposition \ref{prop2.3}.

\medskip
To this end we make some preliminary observations. Let
\begin{equation*}
D = D_1 \cup \dots \cup D_\ell
\end{equation*}
be a partition into $L$-orbits. Then by hypothesis b) of Theorem~\ref{theo0.1} we have
\begin{equation}\label{3.1}
2 \le |D_i|, \quad \forall 1 \le i \le \ell\,.
\end{equation}
The set of $L_n$-orbits in $S_n$ is given by
\begin{equation*}
I_n = \{D_{i_1} \times \dots \times D_{i_n} :(i_1, \dots, i_n) \in [1,\ell]^n\}.
\end{equation*}
Then there are $\ell^n$ orbits of $L_n$ in $S_n$ and for each such orbit $\cO$ we have
\begin{equation}\label{3.2}
2^n \le |\cO| \le |D|^n .
\end{equation}

\medskip\n
{\bf Case 1:} Every pair of orbits $\cO_1 \not= \cO_2$ in $S_n$ leads possibly to a contribution to $\zeta_{L_{n+1}/L_n}(k)$ which is bounded by
\begin{equation*}
\dis\frac{C^2_1 \cdot C_2 \,|D|^n}{(2^{2^n})^k} \,,
\end{equation*}

\n
where $C_1$ is the number of simple quotients of $L$ and $C_2$ an upper bound on the cardinality of their outer automorphism groups; this takes into account the estimate (\ref{3.2}).

\medskip
Summing over all distinct pairs of orbits, the total contribution of Case 1 is bounded by
\begin{equation*}
C^2_1 \, C_2 {\ell^n \choose 2} \; \dis\frac{|D|^n}{(2^{2^n})^k} \le C^2_1 \, C_2 \; \dis\frac{(\ell^2 \,|D|)^n}{(2^{2^n})^k} \;.
\end{equation*}

\medskip\n
{\bf Case 2:}
According to Proposition~\ref{2.3} 2 d) we need to estimate the number of $L_n$-block decompositions of an orbit $\cO \in I_n$. Observe that a $L_n$-block decomposition of $D_{i_1} \times \dots \times D_{i_n}$ is obtained by taking the products of $L$-block decompositions in each $D_{i_j}$. Therefore, if $C_3$ is an upper bound on the number of $L$-block decompositions of $D_i$, $1 \le i \le \ell$, then the number of $L_n$-block decompositions in any orbit $\cO \in I_n$ is bounded by $C_3^n$.

\medskip
Now fix $\cO \in I_n$ such that $B_\cO \not= (e)$. Then every normal subgroup $N \vartriangleleft B$ which is of the form $T^r$ for some non-abelian simple $T$ contributes at most
\begin{equation*}
|{\rm Out}(T)|^{r|\cO |} (2 \,|\cO|^2)^{r-1} \,C_3^{r \cdot \cO} \cdot \big(|N|^{-\frac{|\cO|}{2}}\big)^k
\end{equation*}

\medskip\n
to $\zeta_{L_{n+1} /L_n}(k)$ and this can be rewritten as:
\begin{equation*}
(2 \,|\cO|)^{r-1} \,C_3^{r \cdot n}  \Big(\dis\frac{|{\rm Out} (T)|^r}{|N|^{k/2}}\Big)^{|\cO|} .
\end{equation*}

\medskip\n
Now choose $k$ large enough so that $\frac{|{\rm Out}(T)|^r}{|N|^{k/2}} < 1$, then one sees easily that the sum over all $\cO \in I_n$ and $n \ge 1$ converges, taking into account that $|\cO| \ge 2^n$.

\medskip\n
{\bf Case 3:} The total contribution of this case is easily seen to be bounded by
\begin{equation*}
\dis\frac{C_6\cdot \ell^n}{(2^{2^n})^k}\;,
\end{equation*}
where $C_6$ bounds the number of subgroups of quotients of $L$.

\medskip
Thus the sum over $n \ge 1$ of the contributions to $\zeta_{L_{n+1}/L_n}(k)$ in the Cases 1, 2 and 3 converge already for $k$ large enough depending on $L$.

\medskip
{\bf Case 4:}
 We will treat this case by estimating the number of sections of $L_n$ in $L_n \ltimes B^\cO$. To this end let $\cO = D_{i_1} \times \dots \times D_{i_n}$, we define $\cO_j : = D_{i_1} \times \dots \times D_{i_j}$, $1 \le j \le n$ so that $\cO_n = \cO$. Let $a(j)$ be the number of sections of $L_j$ in $L_j \ltimes B^{\cO_j}$.

\medskip
Let $K$ be an upper bound on the number of non-abelian subgroups of $L  \ltimes B^{D_j}$ and $C_7$ an upper bound on $|{\rm Hom}(L,L \ltimes  B^{D_j})|$ for any quotient $B$ of $L$ and any $1 \le j \le \ell$.

\begin{lemma}\label{lem3.1}
$a(n) \le a(n-1)^{|D_{i_n}|} {|S_{n -1}| \choose K} \;K! \, C_7^K$.
\end{lemma}

\begin{proof}
Let $L_n = L_{n-1} \ltimes  L^{S_{n-1}}$ and represent a given section $\phi$: $L_n \r L_n \ltimes B^{\cO_n}$ by a map of the form
\begin{align*}
\phi: \; L_{n-1} \ltimes  L^{S_{n-1}} & \longrightarrow (L_{n-1} \ltimes L^{S_{n-1}}) \ltimes  B^{\cO_n}
\\[1ex]
(x,y) &  \longrightarrow \big((x,y), f(x,y)\big).
\end{align*}
We analyze separately $\phi|_{L_{n-1}}$ and $\phi |_{L^{S_{n-1}}}$.

\medskip
We have $\phi(x,e) = ((x,e), f(x,e))$ where $f(x,e)$ is a map from ${\cO_n = \cO_{n-1} \times D_{i_n}}$ to $B$. Expressing that $\phi$ is a homomorphism, we see at once that for every $s \in D_{i_n}$, the map
\begin{align*}
L_{n-1} & \longrightarrow L_{n-1} \ltimes  B^{\cO_{n-1}}
\\[1ex]
x &  \longrightarrow \big(x,f(x,e) (\cdot,s)\big)
\end{align*}
is a section of $L_{n-1}$ in $L_{n-1} \ltimes  B^{\cO_{n-1}}$. The number of possible restrictions to $L_{n-1}$ of a section $\phi$ as above is therefore bounded by $a(n-1)^{|D_{i_n}|}$.

\medskip
Next, we analyze $\phi |_{L^{S_{n-1}}}$
\begin{align*}
L^{S_{n-1}} & \longrightarrow L^{S_{n-1}}  \ltimes B^{\cO_{n-1} \times D_{i_n}}
\\[1ex]
y &  \longrightarrow \big(y,f(e,y)\big) .
\end{align*}
Observe that $L^{S_{n-1}} \ltimes B^{\cO_{n-1} \ltimes  D_{i_n}}$ is isomorphic to $L^{S_{n-1} \backslash \cO_{n-1}} \ltimes  (L^{\cO_{n-1}} \ltimes B^{\cO_{n-1} \times D_{i_n}})$ and thus $\phi |_{L^{S_{n-1}}}$ is determined by
\begin{align*}
L^{S_{n-1}} & \longrightarrow L^{\cO_{n-1}}   \ltimes B^{\cO_{n-1} \ltimes D_{i_n}}
\\[1ex]
y &  \longrightarrow \big(y\big|_{\cO_{n-1}}, f (e,y)\big) .
\end{align*}
Now we observe that we have an isomorphism:
\begin{align*}
L^{\cO_{n-1}} \ltimes   B^{\cO_{n-1} \times  D_i} & \longrightarrow (L \ltimes  B^{D_{i_n}})^{\cO_{n-1}}
\\[1ex]
(z,g)&  \longrightarrow (z(\alpha), g_\alpha)_{\alpha \in \cO_{n-1}} ,
\end{align*}

\n
where $g_\alpha(s) : = g(\alpha,s)$, $s \in D_{i_n}$. Thus $\phi |_{L^{S_{n-1}}}$ is determined by a collection of homomorphisms
\begin{equation*}
\phi_\alpha : L^{S_{n-1}} \longrightarrow L \ltimes B^{D_{i_n}},
\end{equation*}

\n
parametrized by $\alpha \in \cO_{n-1}$. But $\phi |_{L^{S_{n-1}}}$ is the restriction of a homomorphism defined on $L_{n-1} \ltimes  L^{S_{n-1}}$ and the $L_{n-1}$-action on $\cO_{n-1}$ is transitive, we conclude that $\phi_\alpha$ is determined once $\phi_{\alpha_0}$ is fixed for some $\alpha_0$. Thus we are left with estimating the number of homomorphisms
\begin{equation*}
L^{S_{n-1}}  \longrightarrow L \ltimes  B^{D_{i_n}} .
\end{equation*}

\n
Given such a homomorphism $\varphi$, for every $\nu \in S_{n-1}$ the image of the corresponding $L$-factor is either trivial or non-abelian (since $L$ is perfect). Since the images via $\varphi$ of different factors commute, they cannot be sent to the same non-abelian subgroup of $L \ltimes B^{D_{i_n}}$.

\medskip
Thus we conclude that the number of such homomorphisms is bounded by
\begin{equation*}
{|S_{n-1}| \choose K} \cdot K! \,C_7^K \,.
\end{equation*}
This concludes the proof of the lemma.
\end{proof}

\begin{lemma}\label{lem3.2}
$a(n) \le C_8^{|\cO_n|}$, where $C_8 = C_7^{\frac{K+1}{2}} |D|^K$.
\end{lemma}

\begin{proof}
Set $|D_{i_j}| = d_j$, the $n$-tuple $(i_1,\dots,i_n)$ being given. In Lemma \ref{lem3.1} estimate
\begin{equation*}
{|S_{n-1}| \choose K}  K!  \le |D|^{K(n-1)}  \,.
\end{equation*}
Iterating the inequality thus obtained:
\begin{equation*}
a(n) \le a(n-1)^{d_n} \,|D|^{K(n-1)} C_7^K
\end{equation*}
we get
\begin{equation*}
a(n) \le a(1)^{d_2 \dots d_n} \,|D|^{K\cdot \alpha_n} (C_7^K)^{\beta_n},
\end{equation*}
where
\begin{align*}
\alpha_n & = d_3 \dots d_n + 2 d_4 \dots d_n + \dots + (n-2) \,d_n + (n-1)
\\
\beta_n & = d_3 \dots d_n + d_4 \dots d_n + \dots  + d_n + 1 .
\end{align*}
Then
\begin{align*}
\alpha_n & = d_3 \dots d_n  \Big[1 + \dis\frac{2}{d_3} + \dis\frac{3}{d_3 \,d_4} + \dots + \dis\frac{n-1}{d_3 \dots d_n}\Big]
\\[1ex]
& \le d_3 \dots d_n  \Big[1 + \dis\frac{2}{2} + \dis\frac{3}{2^2} + \dots + \dis\frac{n-1}{2^{n-2}}\Big]
\\[1ex]
& \le 4d_3 \dots d_n
\intertext{and}
\beta_n & = d_3 \dots d_n  \Big[1 + \dis\frac{1}{d_3} + \dis\frac{1}{d_3 \,d_4} + \dots + \dis\frac{1}{d_3 \dots d_4}\Big]
\\[1ex]
& \le 2d_3 \dots d_n \,.
\intertext{As a result we obtain}
a(n) & \le \big(a(1)^{1/d_1}\big)^{|\cO_n|} \big(|D|^{K/d_1\,d_2}\big)^{4| \cO_n|}  \big(C_7^{K/d_1\,d_2}\big)^{2| \cO_n|}
\\[1ex]
& \le \big(a(1)^{1/2} |D|^K \;C_7^{K/2}\big)^{| \cO_n|}
\\[1ex]
& \le \big(C_7^{(1+ K)/2} |D|^K\big)^{| \cO_n|} \,.
\end{align*}
\end{proof}

\medskip\n
 Every orbit $\cO$ in $I_n$ contributes to $\zeta_{L_{n+1}/L_{n}}(k)$ at most
\begin{equation*}
C_9 \cdot \dis\frac{C_8^{|\cO |}}{(2^{| \cO |})^k} = C_9 \Big(\dis\frac{C_8}{2^k}\Big)^{| \cO |} ,
\end{equation*}

\n
where $C_9$ is an upper bound on the number of quotient groups of $L$. Now choose $k \ge 1$ such that
\begin{equation*}
\dis\frac{C_8}{2^k} < 1.
\end{equation*}
Then the sum over all orbits in $I_n$ of the contribution in Case 4 is bounded by
\begin{equation*}
\ell^n \cdot \Big(\dis\frac{C_8}{2^k} \Big)^{2^n}
\end{equation*}
and with this choice of $k$, the sum over $n \ge 1$ converges.

\medskip
This concludes the proof the theorem.
\nocite{DiMo}
\nocite{Vann}
\nocite{Quick1}
\nocite{Quick2}
\nocite{Bondarenko}
\nocite{Mann1}
\nocite{MaSh}

\bibliographystyle{alpha}
\bibliography{BMbiblio}
\end{document}